\documentclass[12pt]{amsart}

\usepackage[margin=1.15in]{geometry}

\usepackage{amsmath,amscd,amssymb,amsfonts,latexsym,wasysym, mathrsfs, mathtools,hhline,color, bm}
\usepackage[all, cmtip]{xy}

\usepackage{url}

\definecolor{hot}{RGB}{65,105,225}

\usepackage[pagebackref=true,colorlinks=true, linkcolor=hot ,  citecolor=hot, urlcolor=hot]{hyperref}

\usepackage{ textcomp }
\usepackage{ tipa }

\usepackage{graphicx,enumerate}

\theoremstyle{plain}
\newtheorem{theorem}{Theorem}[section]
\newtheorem{proposition}[theorem]{Proposition}

\newtheorem{lm}[theorem]{Lemma}

\newtheorem{conj}[theorem]{Conjecture}
\newtheorem{lemma}[theorem]{Lemma}
\newtheorem{thrm}[theorem]{Theorem}

\theoremstyle{definition}

\newtheorem{definition}[theorem]{Definition}

\newtheorem{remark}[theorem]{Remark}

\newtheorem{ex}[theorem]{Example}
\newtheorem*{ex*}{Example}
\newtheorem{problem}{Problem}

\def\be{\begin{equation}}
\def\ee{\end{equation}}

\def\bt{\begin{thrm}}
\def\et{\end{thrm}}

\def\bc{\begin{cor}}
\def\ec{\end{cor}}

\def\br{\begin{rmk}}
\def\er{\end{rmk}}

\def\bp{\begin{prop}}
\def\ep{\end{prop}}

\def\bl{\begin{lm}}
\def\el{\end{lm}}

\def\bex{\begin{ex}}
\def\eex{\end{ex}}

\def\bd{\begin{defn}}
\def\ed{\end{defn}}

\newcommand\sO{{\mathcal O}}

\newcommand\sS{\mathcal{S}}

\def\cZ{\mathcal{Z}}

\newcommand{\Z}{\mathbb{Z}}

\newcommand{\R}{\mathbb{R}}
\newcommand{\C}{\mathbb{C}}

\renewcommand{\Re}{\operatorname{Re}}
\renewcommand{\Im}{\operatorname{Im}}

\newcommand\pp{{\mathbb{P}}}

\newcommand\rr{{\mathbb{R}}}
\newcommand\cc{{\mathbb{C}}}

\DeclareMathOperator{\reg}{reg}                  
                  
\DeclareMathOperator{\Bl}{bl}

\DeclareMathOperator{\EDdeg}{EDdeg}

\def\bC{\mathbb{C}}
\def\RR{\mathbb{R}}

\def\bP{\mathbb{P}}

\def\bQ{\mathbb{Q}}

\def\balpha{{\bm\alpha}}
\def\bbeta{{\bm\beta}}

\DeclareMathOperator{\Jac}{{Jac}}

\title[ED degree of the multiview variety]{Euclidean distance degree of the multiview variety}

\author{Laurentiu G. Maxim}
\address{Department of Mathematics,         University of Wisconsin-Madison,  480 Lincoln Drive, Madison WI 53706-1388, USA.}
\email {maxim@math.wisc.edu}\urladdr{https://www.math.wisc.edu/~maxim/}
\author{Jose Israel Rodriguez}
\address{Department of Mathematics,         University of Wisconsin-Madison,  480 Lincoln Drive, Madison WI 53706-1388, USA.}
\email {jose@math.wisc.edu}\urladdr{http://www.math.wisc.edu/~jose/}
\author{Botong Wang}
\address{Department of Mathematics,         University of Wisconsin-Madison,  480 Lincoln Drive, Madison WI 53706-1388, USA.}
\email {wang@math.wisc.edu}\urladdr{http://www.math.wisc.edu/~wang/}



\begin{document}

\date{\today}

\maketitle

\begin{abstract}  
The Euclidean distance degree of an algebraic variety is a well-studied topic in applied algebra and geometry. It has direct applications in geometric modeling, computer vision, and statistics. We use non-proper Morse theory to give a topological interpretation of the Euclidean distance degree of an affine variety in terms of Euler characteristics. As a concrete application, we solve the open problem in computer vision of determining the Euclidean distance degree of the affine multiview variety. 
\end{abstract}



\section{Introduction}\label{intro}
To any $\bm{\alpha}=(\alpha_1, \ldots, \alpha_n)\in \bC^n$, one associates the \emph{squared Euclidean distance function}  
 $f_{\bm{\alpha}}: \bC^n\to \bC$ given by
$$f_{\bm{\alpha}}(z_1,\ldots,z_n):=\sum_{1\leq i\leq n}(z_i-\alpha_i)^2.$$
If $X$ is an irreducible closed subvariety of $\bC^n$ then, for generic choices of $\bm{\alpha}$, the function $f_{\bm{\alpha}}|_{X_{\reg}}$ has finitely many critical points on the smooth locus $X_{\reg}$ of $X$. Moreover, this number of critical points is independent of the generic choice of $\bm{\alpha}$, so it defines an intrinsic invariant of $X$, called the {\it Euclidean distance degree} (or ED degree) of $X$; see \cite{DHOST}. It is denoted by $\EDdeg(X)$. 

The motivation for studying ED degrees comes from the fact that many models in data science or engineering are realized as real algebraic varieties, for which one needs to solve a {\it nearest point problem}. Specifically, for such a real algebraic variety $X \subset \RR^n$, one needs to solve the following:

\begin{problem}[Nearest point]\label{problem:real}
Given $\balpha\in\RR^n$, compute $\balpha^*\in X_{\reg}$ that minimizes the squared Euclidean distance 
from the given point $\balpha$.
\end{problem}

When a solution to Problem~\ref{problem:real}
exists, computing all of the critical points of $f_{\bm{\alpha}}$ on the Zariski closure of $X$ in $\bC^n$ will provide one way to find the answer. 
Thus, the ED degree gives an algebraic measure of complexity of this problem.

\medskip 

This paper deals with a very specific nearest point problem, motivated by the {\it triangulation problem} in computer vision and the {\it multiview conjecture} of \cite[Conjecture 3.4]{DHOST}. In computer vision, triangulation refers to the process of reconstructing a point in 3D space from its camera projections onto several images. The triangulation problem has many practical applications, e.g.,  in tourism, for reconstructing the 3D structure of a tourist attraction based on a large number of online pictures \cite{AS}; in robotics, for creating a virtual 3D space from multiple cameras mounted on an autonomous vehicle;  in filmmaking, for  adding animation and graphics to a movie scene after everything is already shot, etc.

The triangulation problem is in theory trivial to solve: if the image points are given with infinite precision, then two cameras suffice to determine the 3D point. In practice, however, various sources of ``noise'' (such as pixelation or lens distortion) lead to inaccuracies in the measured image coordinates. The problem, then, is to find a 3D point which optimally fits the measured image points.

A 3D world point gives rise to $n$ 2D projections in $n$ given cameras. Roughly speaking, the space of all possible $n$-tuples of such projections is the {\it affine multiview variety} $X_n$; see {\cite[Example 3.1]{DHOST}} and Section \ref{cv} for more details. The above optimization problem translates into finding a point $\balpha^*\in X_n$ of minimum distance to a (generic) point $\balpha \in \mathbb{R}^{2n}$ obtained by collecting the 2D coordinates of $n$ noisy images of the given 3D point. In order to find such a minimizer algebraically, one regards $X_n$ as a complex algebraic variety and examines all complex critical points of the squared Euclidean distance function $f_{\balpha}$ on $X_n$. 
Since by construction the complex algebraic variety $X_n$ is smooth,
one is therefore interested 
in computing the Euclidean distance degree $\EDdeg(X_n)$ of the affine multiview variety $X_n$.

An explicit conjectural formula for the Euclidean distance degree $\EDdeg(X_n)$ was proposed in \cite[Conjecture 3.4]{DHOST}, based on numerical computations from \cite{SSN} for configurations involving $n \leq 7$ cameras:
\begin{conj}\label{cedc}\cite{DHOST} The Euclidean distance degree of the affine multiview variety $X_n$ is given by:
\be\label{edc}
\EDdeg(X_n)=\frac{9}{2}n^3-\frac{21}{2}n^2+8n-4.
\ee
\end{conj}

It was recently shown in \cite{HL} that $$\EDdeg(X_n) \leq 6n^3-15n^2+11n-4,$$ for any $n \geq 2$. 

In this paper we give a proof of Conjecture \ref{cedc}. In order to achieve this task, we first interpret the ED degree in terms of an Euler-Poincar\'e characteristic, as follows (see Theorem~\ref{cor_smooth}):
\bt\label{edec}
Suppose $X$ is a smooth closed subvariety of $\bC^n$. Then for general $\bbeta=(\beta_0,\ldots,\beta_n)\in \bC^{n+1}$ we have:
\be\label{e3i} \EDdeg(X)=(-1)^{\dim X}\chi\big(X\cap U_\bbeta\big),\ee
where $U_\bbeta:=\bC^n\setminus \{\sum_{1\leq i\leq n}(z_i-\beta_i)^2+\beta_0=0\}$. 
\et 

Given formula (\ref{e3i}) and the fact that the affine multiview variety $X_n$ is smooth, we prove Conjecture \ref{cedc} by computing the Euler-Poincar\'e characteristic $\chi\big(X_n \cap U_\bbeta\big)$ for a generic $\bbeta \in \bC^{n+1}$. This is done in this paper by regarding the affine multiview variety $X_n$ as a Zariski open subset in its closure $Y_n$ in $(\bP^2)^n$, and using additivity properties of the Euler-Poincar\'e characteristic together with a detailed study of the topology of the divisor $Y_n \setminus X_n$ ``at infinity''.

Theorem \ref{edec} is proved using Morse theory. 
We study real Morse functions  of the form $\log |f|$, where $f$ is a {nonvanishing} holomorphic Morse function on a complex manifold.
Such a Morse function has the following important properties: 
\begin{enumerate}
\item The critical points of $\log |f|$ coincide with the critical points of $f$.
\item The index of every critical point of $\log |f|$ is equal to the complex dimension of the manifold  on which $f$ is defined.
\end{enumerate}
However, as a real-valued Morse function $\log|f|$ is almost never proper. So we use here the non-proper Morse theory techniques developed by Palais-Smale \cite{PS64} (see also~\cite{LMW}). 

Alternatively one can derive 
Theorem \ref{edec} (and its generalization to singular varieties as in Theorem~\ref{csingi} below)
 by using more general results from stratified Morse theory as in \cite{STV}, \cite[Chapter 6]{Tib}, or \cite{TS}.
However, our proof of Theorem \ref{edec} is more elementary and it suffices to prove Conjecture~\ref{cedc} which motivated this work.
Moreover, it may also be of independent interest. 

In the presence of singularities, formula \eqref{e3i} is no longer true. 
Instead,
one needs to replace the Euler-Poincar\'e characteristic on the right hand side by the Euler characteristic of the local Euler obstruction function to capture the topology of singularities. 
More precisely, the results in  \cite{STV}, \cite[Chapter 6]{Tib}, or \cite{TS} can be used to prove 
 the following generalization of Theorem~\ref{edec} to the singular setting (see Theorem~\ref{csing}):

\bt\label{csingi} Let $X$ be an irreducible closed subvariety of $\C^n$. Then for a general $\bbeta\in \C^{n+1}$ we have:
\be\label{geni}
\EDdeg(X)=(-1)^{\dim X}\chi({\rm Eu}_X|_{U_{\bbeta}}),
\ee
where ${\rm Eu}_X$ is the local Euler obstruction function on $X$.
\et

With formula (\ref{geni}) one can compute the ED degree of a singular variety by understanding  the Euler characteristic of the local Euler obstruction function on $X$. 
The ED degree of singular varieties comes up in many instances in applications, the most familiar being the Eckhart-Young theorem
(see \cite[Example 2.3]{DHOST}) which is used in low-rank approximation.


 
 \medskip

The paper is organized as follows. 
In Section~\ref{npr}, we introduce the necessary material on the non-proper Morse theory of Palais-Smale and explain a holomorphic analogue of this theory. In Section~\ref{topaff}, we apply the results from Section \ref{npr} to study the topology of smooth affine varieties (Theorem~\ref{theorem_top}). As a further application, we prove Theorem~\ref{edec} on a topological formula for the Euclidean distance degree of smooth affine varieties. We also indicate here how results of \cite{STV} and \cite{TS} can be employed to  compute the Euclidean distance degree of any affine variety. Section \ref{cv} is devoted to the proof of Conjecture \ref{cedc}. 

\medskip 

\noindent{\bf Acknowledgements.} The authors thank Alex Dimca, Jan Draisma, and Lei Wu for useful discussions, and 
we are especially grateful to 
J\"org Sch\"urmann for reading earlier versions of the manuscript and making several suggestions for improvement. 
L. Maxim is partially supported by the Simons Foundation Collaboration Grant \#567077 and by the
Romanian Ministry of National Education, CNCS-UEFISCDI, grant PN-III-P4-ID-PCE-2016-0030. 
J.~I. Rodriguez is partially supported by the College of Letters and Science, UW-Madison. 
B. Wang is partially supported by the NSF grant DMS-1701305.



\section{Non-proper Morse theory}\label{npr}
Classical Morse theory (e.g., see \cite{Mi}) relates the number of critical points of a Morse function on a space with the topology of the space. In this paper, we are interested in real-valued Morse functions that are induced from holomorphic Morse functions. 

\subsection{Non-proper Morse theory and Palais-Smale conditions}
First, we recall the definition of Morse functions. 
\begin{definition}
Let $M$ be a smooth manifold. A smooth function $f: M\to \R$ is a \emph{real-valued Morse function} if all of its critical points are non-degenerate\footnote{In literature, Morse functions are often required to be proper and to have distinct critical values. In this paper, we deal with non-proper functions. Since we use Morse theory only to count the number of cells attached, we also do not require the critical values to be distinct. In fact, the critical values can be made distinct by a perturbation argument, while keeping the critical points unchanged.}.
\end{definition}

\begin{definition}
Let $M$ be a complex manifold. A holomorphic function $f: M\to \C$ is a \emph{holomorphic Morse function} if all of its critical points are non-degenerate. 
A $\C^*$-valued holomorphic function is a \emph{$\C^*$-valued holomorphic Morse function} if it is a holomorphic Morse function when regarded as a $\C$-valued function. 
\end{definition}

Let $f: M\to \R$ be a 
smooth function on a manifold $M$. For any real numbers $a<b$, we define $f^{a,b}:=f^{-1}\big([a, b]\big)$ and $f^a:=f^{-1}\big((-\infty, a]\big)$. The following  result is due to Palais-Smale\footnote{The proof was skipped in the original paper of Palais-Smale \cite{PS64}, but a proof is sketched in \cite[Theorem 3.1]{LMW}.}: 
\begin{theorem}[\cite{PS64}]\label{theorem_nonproper}
Let $M$ be a complete Riemannian manifold, and let $f: M\to \R$ be a real-valued Morse function satisfying the following condition: 
\begin{itemize}
\item[\textbf{(PS1)}]\label{condition_real} If $S$ is a subset of $M$ on which $|f|$ is bounded but on which $||\nabla f||$ is not bounded away from zero, then there exists a critical point of $f$ in the closure of $S$. 
\end{itemize}
Then the following properties hold: 
\begin{enumerate}
\item For any real numbers $a<b$, there are finitely many critical points of $f$ in $f^{a,b}$.
\item Let $a, b$ be regular values of $f$. Suppose that there are $r$ critical points of $f$ in $f^{a, b}$ having index $d_1, \ldots, d_r$, respectively. Then $f^{b}$ has the homotopy type of $f^{a}$ with $r$ cells of dimensions $d_1, \ldots, d_r$ attached. 
\item If $c$ is a regular value of $f$, then $f$ has the structure of a fiber bundle in a small neighborhood of $c$. 
\end{enumerate}
\end{theorem}




In \cite[Theorem 3.2]{LMW}, a circle-valued analogue of Theorem \ref{theorem_nonproper} was proved. In this paper, we develop holomorphic versions of the Palais-Smale condition (PS1).

\subsection{Holomorphic Palais-Smale conditions}
Let $M$ be a complex manifold with a complete Hermitian metric $h$. Let $f: M\to \C$ be a holomorphic Morse function on $M$. 
We introduce the following Palais-Smale condition of a $\C$-valued holomorphic Morse function:
\begin{itemize}
\item[\textbf{(PS2)}]\label{condition_complex1} If $S$ is a subset of $M$ on which $|f|$ is bounded but on which $||\nabla f||_h$ is not bounded away from zero, then there exists a critical point of $f$ in the closure of $S$. 
\end{itemize}

Similarly, we can define a Palais-Smale condition for a $\C^*$-valued function. Let $f: M\to \C^*$ be a holomorphic Morse function on $M$. We call the following the Palais-Smale condition of a $\C^*$-valued holomorphic Morse function:
\begin{itemize}
\item[\textbf{(PS3)}]\label{condition_complex2} If $S$ is a subset of $M$ on which $\log |f|$ is bounded but on which $||\nabla \log f||_h$ is not bounded away from zero, then there exists a critical point of $f$ in the closure of $S$. 
\end{itemize}
Here we notice that even though $\log f$ is a multivalued funtion, $d\log f=\frac{df}{f}$ is well-defined, and hence $\nabla \log f$ is well-defined.

\begin{lemma}\label{lemma_equiv1}
	Let $M$ be a complex manifold with a complete Hermitian metric $h$. Denote the associated Riemannian metric of $h$ by $g$, i.e., $g$ is the real part of $h$. Let $f: M\to \C$ be a holomorphic function. Then 
	$$||\nabla f||_h=\sqrt{2}\cdot ||\nabla f_1||_g$$
	where $f_1=\Re f$ is the real part of $f$.
\end{lemma}
\begin{proof} 
Notice that a Riemannian (resp., Hermitian) metric on a real (resp., complex) vector space induces a Riemannian (resp., Hermitian) metric on the dual vector space. Thus, we can also consider $h$ and $g$ as metrics on the real and complex cotangent vector bundle of $M$, respectively. Then, by the definition of gradient, we have
$$||\nabla f||_h=||df||_h \text{ and } ||\nabla f_1||_g=||df_1||_g.$$
Since $g$ is the associated Riemannian metric of $h$, we have
\begin{equation}\label{eq_HR1}
||v||_h^2=<v, v>_h=<\Re v, \Re v>_g+<\Im v, \Im v>_g=||\Re v||_g^2+||\Im v||_g^2
\end{equation}
for any complex cotangent vector $v$ at any point on $M$. Denote  $\Im f$ by  $f_2$.  By the Cauchy-Riemann equation, we have
\begin{equation}\label{eq_HR2}
||df_1||_g=||df_2||_g
\end{equation}
at every point of $M$. By equations (\ref{eq_HR1}) and (\ref{eq_HR2}), we have $||df||_h=\sqrt{2}||df_1||_g$. 
\end{proof}

\begin{lemma}\label{lemma_41}
Let $M$ be a complex manifold with a Hermitian metric $h$. 
If $f: M\to \C^*$ is a $\C^*$-valued holomorphic Morse function satisfying the Palais-Smale condition (PS3), 
then $\log |f|$ is a real-valued Morse function satisfying the Palais-Smale condition (PS1).
Moreover, $f$ and $\log|f|$ have the same critical points,
and the index of all critical points of $\log|f|$ is the complex dimension of $M$.
\end{lemma}

\begin{proof}
	In the notations of Lemma \ref{lemma_equiv1}, it suffices to prove that for any subset $S$ of $M$, the function $||d\log f||_h$ is bounded away from zero on $S$ if and only if $\big|\big|d\log |f|\big|\big|_g$ is bounded away from zero on $S$. 

    Let $\tilde{f}: \widetilde{M}\to \C$ be the lifting of $f: M\to \C^*$ by the exponential map $\exp: \C\to \C^*$ defined by $\exp(z)=e^{z}$. Then there is a  natural infinite cyclic covering map $\pi:\widetilde{M}\to M$, which induces a Hermitian metric $\tilde{h}$ on $\widetilde{M}$. Denote the associated Riemannian metric by $\tilde{g}$. 
    Let $\widetilde{S}=\pi^{-1}(S)$. We claim that the following statements are equivalent: 
    \begin{enumerate}
		\item\label{equiv1} On $S$, the function $||d \log f||_h$ is  bounded away from zero.
		\item\label{equiv2} On $\widetilde{S}$, the function $||d \tilde{f}||_{\tilde{h}}$ is  bounded away from zero. 
		\item\label{equiv3} On $\widetilde{S}$, the function $||d \Re \tilde{f}||_{\tilde{g}}$ is  bounded away from zero.  
		\item\label{equiv4} On $S$, the function $\big|\big|d\log |f|\big|\big|_g$ is  bounded away from zero. 
	\end{enumerate}
    The equivalences $(\ref{equiv1})\iff(\ref{equiv2})$ and $(\ref{equiv3})\iff(\ref{equiv4})$ follow immediately from the construction. The equivalence $(\ref{equiv2})\iff(\ref{equiv3})$ follows from Lemma \ref{lemma_equiv1}. 
Thus, the function $\log |f|$ satisfies (PS1).
Notice that $\log |f|=\Re \log f$,
so by the Cauchy-Riemann equations the critical points of $f$ are the same as the critical points of $\log |f|$.
Moreover, by the holomorphic Morse lemma \cite[Section 2.1.2]{Voisin2002}, 
the index of each critical point of $\log |f|$ is equal to the complex dimension of $M$.
\end{proof}

\begin{lemma}\label{lemma_restriction}
Let $M$ be a complex manifold with a Hermitian metric $h$. Suppose $f: M\to \C$ is a holomorphic Morse function satisfying the $\C$-valued Palais-Smale condition (PS2). Then, for any $t\in \C$, the restriction $f|_{U_t}: U_t\to \C^*$ satisfies the $\C^*$-valued Palais-Smale condition (PS3), with $U_t:=f^{-1}(\C\setminus\{t\})$ endowed with the same metric $h$. 
\end{lemma}
\begin{proof}
	The $\C^*$-valued Palais-Smale condition (PS3) for $f|_{U_t}$ states that there does not exists $S\subset U_t$ such that:
	\begin{enumerate}
		\item $\log |f-t|$ is bounded on $S$;
		\item $||d\log (f-t)||_h$ is not bounded away from zero on $S$;
		\item $f$ has no critical point in the closure of $S$ in $U_t$. 
	\end{enumerate}
    On the other hand, the holomorphic Palais-Smale condition (PS2) for $f$ states that there does not exists $S\subset M$ such that:
    \begin{enumerate}
    	\item[(4)] $|f|$ is bounded;
    	\item[(5)] $||df||_h$ is not bounded away from zero on $S$;
    	\item[(6)] $f$ has no critical point in the closure of $S$ in $M$. 
    \end{enumerate}
    Notice that $d\log (f-t)=\frac{df}{f-t}$. When $\log |f-t|$ is bounded, $||df||_h$ is not bounded away from zero on $S$ if and only if $||d\log (f-t)||_h=||\frac{df}{f-t}||_h$ is not bounded away from zero on $S$. Thus, $(1)+(2)\Rightarrow (5)$. Evidently, $(1)\Rightarrow (4)$ and $(3)\Rightarrow (6)$. Therefore, $(1)+(2)+(3)\Rightarrow (4)+(5)+(6)$, and the assertion in the theorem follows. 
\end{proof}

The following result is analogous to Theorem 4.10 of \cite{LMW}, and the proof is essentially the same. 

\begin{theorem}\label{theorem_linear}
Let $X$ be a smooth closed subvariety of $\C^n$ with the induced Hermitian metric from the Euclidean metric on $\C^n$. Then for a general linear function $l$ on $\C^n$, its restriction $l|_X$ to $X$ is a holomorphic Morse function satisfying the $\C$-valued Palais-Smale condition (PS2). 
\end{theorem}

\begin{proof}
Denote the complex vector space $\C^n$ by $V$, and denote its dual vector space by $V^\vee$. We define a closed subvariety $\cZ\subset V\times V^\vee$ by
	$$\cZ=\{(x, l)\in V\times V^\vee \mid x \text{ is a critical point of } l|_X \}.$$
	If we consider $V\times V^\vee$ as the cotangent space of $V$, then $\cZ$ is equal to the conormal bundle of $X$ in $V$. Now, consider the second projection $p: V\times V^\vee\to V^\vee$. Since $\dim \cZ=\dim V^\vee=n$, the restriction $p|_\cZ: \cZ\to V^\vee$ is generically finite. Therefore, there exists a nonempty Zariski open subset $U\subset V^\vee$ on which the map $p|_\cZ$ is finite and \'etale, that is, a finite (unramified) covering map (e.g., see \cite[Corollary 5.1]{Ve76}). 
	
	We will show that for any $l\in U$, its restriction $l|_X$ is a holomorphic Morse function satisfying the $\C$-valued Palais-Smale condition (PS2). So let us fix $l\in U$.
	
	First, there exists a bijection between the intersection $\cZ\cap p^{-1}(l)$ and the critical points of $l|_X$. Moreover, a critical point of $l|_X$ is non-degenerate if and only if the  intersection $\cZ\cap p^{-1}(l)$ is transverse at the corresponding point. By our construction, $p|_\cZ: \cZ\to V^\vee$ is \'etale at $l$, which means that the intersection $\cZ\cap p^{-1}(l)$ is transverse. Therefore, $l|_X$ is a holomorphic Morse function on $X$. 
	
	Next, we prove that $l|_X$ satisfies the Palais-Smale condition (PS2). Suppose $x_i\in X$, for $i=1, 2, \ldots$, is a sequence of points in $S$ such that the sequence $||\nabla (l|_X)_{x_i}||_h$ converges to zero. By \cite[Lemma 6.2]{LMW}, there exist $l_i\in V^\vee$ such that $||l_i-l||_h=||\nabla (l|_X)_{x_i}||_h$, and $(l_i)|_X$ has a critical point at $x_i$. Here the metric on $V^\vee$ is the induced Hermitian metric from the standard Euclidean Hermitian metric on $V=\C^n$. Then we have a sequence $(x_i, l_i)\in \cZ$ with $l_i$ converging to $l$ in $V^\vee$. Since the map $p|_\cZ: \cZ\to V^\vee$ is finite and \'etale near $l\in V^\vee$, there exists a subsequence of $(x_i, l_i)$ converging to a point in $V\times V^\vee$ (see \cite[Lemma 6.1]{LMW} for a precise proof of this fact). The limit point is of the form $(x_0, l)$. Since $\cZ$ is closed in $V\times V^\vee$, we have $(x_0, l)\in \cZ$. Hence, the point $x_0$ is a critical point of $l|_X$, which is in the closure of $S$.  
\end{proof}

\begin{remark}
In fact, the proof of Theorem \ref{theorem_linear} shows that $l|_X$ satisfies a stronger condition than (PS2), where we do not require that $|l|_X|$ is bounded on $S$. In other words, we have shown that if $S$ is a subset of $X$ on which $||\nabla l|_X||$ is not bounded away from zero, then there exists a critical point of $l|_X$ in the closure of $S$. Now, by Lemma \ref{lemma_equiv1}, for a general linear function $l: \C^n\to \C$, the real valued function $\Re l|_X: X\to \R$ satisfies (PS1). 
\end{remark}
\begin{remark}\label{remark_explanation}
	For a smooth affine variety $X\subset \C^n$, the holomorphic Palais-Smale condition (PS2) for $l|_X$ essentially means that fiberwise there is no critical point at  infinity. Given a linear map $l: \C^n\to \C$, we can compactify fiberwise and obtain a proper map $\overline{l}: \mathbb{P}^{n-1}\times \C\to \C$, extending $l$. Let $\overline{X}$ be the closure of $X$ in $\mathbb{P}^{n-1}\times \C$. Suppose $\overline{X}$ is also smooth. Then $l|_X$ satisfies (PS2) if and only if the function $l_{\overline{X}}: \overline{X}\to \C$ has no critical point in the boundary $\overline{X}\setminus X$. 
\end{remark}

\begin{remark}
The proof of Theorem~\ref{theorem_linear} shows that the number of critical points of $l|_{X}$ is the degree of the projection $p|_\cZ: \cZ\to V^\vee$.
\end{remark}


\section{Topology of affine varieties and the Euclidean distance degree}\label{topaff}
In this section, we apply the results from Section~\ref{npr} to study the topology of smooth affine varieties. As a further application, we derive a topological formula for the Euclidean distance degree of smooth affine varieties. Moreover, by using results of \cite{STV,TS}, we also show how to compute the Euclidean distance degree of a possibly singular affine variety in terms of the local Euler obstruction function.

\subsection{Euler-Poincar\'e characteristic and the number of critical points}
 
\begin{theorem}\label{theorem_top}
Let $X$ be a smooth closed subvariety of dimension $d$ in $\C^n$. Let 
$l: \C^n\to \C$ be a general linear function, and let $H_c$ be the hyperplane in $\C^n$ defined by the equation $l=c$ for a general $c\in \C$. Then:
\begin{enumerate}
	\item $X$ is homotopy equivalent to $X\cap H_c$ with finitely many $d$-cells attached;
	\item the number of $d$-cells is equal to the number of critical points of $l|_X$;
	\item the number of critical points of $l|_X$ is equal to $(-1)^{d}\chi(X\setminus H_c)$. 
\end{enumerate}
\end{theorem}
\begin{proof} Let $\mathcal{B}_X$ denote the bifurcation set for the  function $l|_X:X \to \C$. It is well-known that $\mathcal{B}_X$ is finite and that $l|_X$ is a smooth fiber bundle over $\C \setminus \mathcal{B}_X$. Let $c \in \C \setminus \mathcal{B}_X$, and denote by  
 $B_\epsilon(c)$ the disc in $\C$ centered at $c$ with radius $\epsilon>0$.  
	Since $c$ is not a bifurcation point, the map $l|_X: X\to \C$ is a fiber bundle near $c\in \C$. Thus, for $\epsilon$ sufficiently small, $l|_X^{-1}(c)$ is a deformation retract of $l|_X^{-1}\big(B_\epsilon(c)\big)$. 
	
	By Theorem \ref{theorem_linear} and Lemma \ref{lemma_restriction}, the function $l|_{U_c}: U_c\to \C^*$ is a $\C^*$-valued holomorphic Morse function satisfying the Palais-Smale condition (PS3), where $U_c:=X\setminus H_c$ is endowed with the Euclidean Hermitian metric induced from $\C^n$. By Lemma \ref{lemma_41}, the function $\log \big|l|_{U_c}\big|: U_c\to \R$ is a real-valued Morse function satisfying the condition (PS1). 
	Moreover, the holomorphic function $l|_{U_c}$ and the real-valued function $\log \big|l|_{U_c}\big|$ have the same critical points and the index of all critical points of $\log \big|l|_{U_c}\big|$ is exactly $d$.
	%
%
	  It then follows from Theorem \ref{theorem_nonproper} that, for $\epsilon>0$ sufficiently small, $l|_X^{-1}\big(B_{\frac{1}{\epsilon}}(c)\big)$ is homotopy equivalent to $l|_X^{-1}\big(B_\epsilon(c)\big)$ with finitely many $d$-cells attached. Moreover, the number of $d$-cells is equal to the number of critical points of $l|_{U_c}$, which is equal to the number of critical points of $l|_X$. For $\epsilon>0$ sufficiently small so that $\mathcal{B}_X \subset B_{\frac{1}{\epsilon}}(c)$, we have that $l|_X^{-1}\big(B_{\frac{1}{\epsilon}}(c)\big)$ is homotopy equivalent to $X$. 
	 Thus, the first two assertions follow from Theorem \ref{theorem_nonproper}. 
	
	By the additivity of the Euler characteristics for complex algebraic varieties, we have 
	$$\chi(X\setminus H_c)=\chi(X)-\chi(X\cap H_c)=(-1)^{d}\cdot (\text{the number of critical points of } l|_X),$$
	where the second equality follows from assertions (1) and (2).
\end{proof}

\begin{remark}\label{fibr} The above proof also shows that, for general $c \in \C$, the complement $U_c=X\setminus H_c$ of a generic hyperplane section of $X$ is obtained (up to homotopy) from a fiber bundle over $S^1\simeq B_\epsilon(c) \setminus \{c\}$ with fiber homotopy equivalent to a finite CW-complex, by attaching finitely many $d$-cells. In particular, it follows as in \cite[Section 2]{LMW} that $U_c$ satisfies a weak form of generic vanishing for rank-one local systems, and the number of $d$-cells (hence also the number of critical points of $l|_X$) equals (up to a sign) the middle Novikov (or $L^2$) Betti number of $U_c$ corresponding to the homomorphism $l_*:\pi_1(U_c)\to \Z$ induced by $l$.
\end{remark}

\begin{remark}
The assertion (1) of Theorem~\ref{theorem_top} is a special case of the {\it affine Lefschetz theorem}, see \cite[Theorem 5]{Hamm}.	
\end{remark}


\subsection{Examples}
When $V=\C^n$ and $V\times V^\vee$ is the cotangent space of $V$, 
the \emph{conormal  variety} of a smooth codimension $c$ variety $X$ in $V$ has defining equations given by the following. 
Let $(z_1,\dots,z_n)$ denote the coordinates of $V$, $(u_1,\dots,u_n$) denote the coordinates of $V^\vee$, and let $\{f_1,\dots,f_k\}$ denote a set of polynomials generating the ideal of $X$. 
If $\Jac (f_1,\dots,f_k)$ denotes the $k\times n$ matrix of partial derivatives where the $(i,j)$th entry is given by $\partial f_i/\partial z_j$, then the ideal of the conormal variety is generated by the sum of 
$\langle  f_1,\dots,f_k\rangle$ and the $(c+1)\times (c+1)$ minors of  the $(k+1)\times n$ matrix
$\begin{bmatrix}u_1,\dots,u_n\\  
\Jac (f_1,\dots,f_k)
\end{bmatrix}.
$
Readers with an optimization background
can compare these defining equations to those of the conormal variety of a projective variety presented in Chapter 5 of \cite[p.~215]{BPT2013}. 

\bex
	Let $X\subset V=\C^2$ be the zero locus of $f(x, y)=y-x^2$. Let $u, v$ be the dual coordinates in $V^\vee$. 
	
	The conormal variety $\cZ\subset V\times V^\vee$ is defined by equations 
	$$y=x^2, \text{ and } u\frac{\partial f}{\partial y}-v\frac{\partial f}{\partial x}=0. $$
	In this case, the second projection $p|_\cZ: \cZ\to V^\vee$ is a birational map. It is finite and \'etale, i.e., an isomorphism in this case, over $U=\{v\neq 0\}$. Therefore, for any linear function of the form $l=\alpha x+ y$, the function $l|_X$ is a holomorphic Morse function satisfying (PS2). 
	
	When $l=\alpha x+ y$, the function $l|_X$ has one critical point. On the other hand, since $X\cong \C$ and $X$ has degree two, we have 
	$$\chi(X\setminus H_c)=\chi(X)-\chi(X\cap H_c)=1-2=-1.$$
	
	When $l=x$, the function $l|_X$ happens to also satisfy (PS2). This is explained in Remark \ref{remark_explanation}. 
\eex

\bex
     Let $X\subset V=\C^2$ be defined by the equation $x(x+1)y=1$. Let $u, v$ be the dual coordinates in $V^\vee$. Then the conormal variety $\cZ\subset V\times V^\vee$ is defined by equations
     $$x(x+1)y=1, \text{ and } (2x+1)yv-x(x+1)u=0.$$
     The second projection $p|_\cZ: \cZ\to V^\vee$ is of degree four. The map is finite\footnote{The computation of where the map is finite but not \'etale is a bit complicated and not important for our purpose. } over $U'=\{v\neq 0\}$. 
     
     For a general linear function $l=\alpha x+\beta y$, the restriction $l|_X$ is equal to $\alpha x+\frac{\beta}{x(x+1)}$ on $X$. Since $x$ is a coordinate function of $X$ with $x\neq 0, 1$, we know that $l|_X$ has four critical points. On the other hand, $X$ is a curve of degree three, which is isomorphic to $\C\setminus \{0, 1\}$. Thus, $\chi(X\setminus H_c)=\chi(X)-\chi(X\cap H_c)=-1-3=-4$. 
     
     The linear function $l=x$ has no critical point on $X$. The intersection $X\cap H_c$ consists of one or zero points depending on whether $c(c-1)=0$. Thus, 
     $$\chi(X\setminus H_c)=\chi(X)-\chi(X\cap H_c)=-1 \text{ or }-2.$$
     The number of critical points of $l|_X$ is not equal to $(-1)^{\dim X}\chi(X\setminus H_c)$ in either case. We can also see that (PS2) fails in this case, because as $x\to 0$ and $y\to \infty$ in $X$, the function $l|_X$ is bounded and the norm $||dl|_X||\to 0$. 
\eex

\bex
	Let $X$ be a curve in $\C^2$ defined by a general degree $d$ polynomial $f(x, y)=0$. In particular, $X$ is smooth. For a general linear function $l=\alpha x+\beta y$ on $\C^2$, then the critical points of $l|_X$ are defined by 
	$$f(x, y)=0, \text{ and } \alpha \frac{\partial f}{\partial y}-\beta \frac{\partial f}{\partial x}=0.$$
	By Bezout's theorem, $l|_X$ has $d(d-1)$ critical points. 
	
	On the other hand, the compactification $\overline{X}$ is a smooth curve of genus $g=\frac{(d-1)(d-2)}{2}$. Thus, $\chi(X)=(2-2g)-d=2d-d^2$, and $\chi(X\setminus H_c)=\chi(X)-\chi(X\cap H_c)=d-d^2$. 
	
\eex

\bex
	Let $X$ be a smooth closed subvariety of $\C^n$ whose defining equation does not involve the last coordinate of $\C^n$. Then any linear function $l$ on $\C^n$ involving the last coordinate has no critical point on $X$. On the other hand, the projection on the first $n-1$ coordinates induces a $\C^*$-bundle structure on $X\setminus H_c$. Thus, $\chi(X\setminus H_c)=0$. 
\eex


\subsection{Euclidean distance degree}
In this section we give a topological interpretation of the Euclidean distance degree of an affine variety in terms of an Euler characteristic invariant. (Other such interpretations can be derived in the smooth setting by using Remark \ref{fibr}.) We begin with the following application of Theorem~\ref{theorem_top}:
\begin{theorem}\label{cor_smooth}
	Let $X$ be a smooth closed subvariety of $\C^n$, and let $z_1, \ldots, z_n$ be the coordinates of $\C^n$. For a general $\bbeta=(\beta_0, \ldots, \beta_{n})\in \C^{n+1}$, let $U_{\boldsymbol{\beta}}$ denote the complement of the hypersurface $\sum_{1\leq i\leq n}(z_i-\beta_i)^2+\beta_{0}=0$ in $\C^n$. Then 
	\begin{equation}
	\EDdeg(X)=(-1)^{\dim X}\chi(X\cap U_{\bbeta}). 
	\end{equation}
\end{theorem}
\begin{proof}
	Consider the closed embedding 
	$$i: \C^n\hookrightarrow \C^{n+1};\quad (z_1, \ldots, z_n)\mapsto ( z_1^2+\cdots+z_n^2, z_1, \ldots, z_n).$$
	Let $w_0, \ldots, w_{n}$ be the coordinates of $\C^{n+1}$. Notice that function $\sum_{1\leq i\leq n}(z_i-\beta_i)^2+\beta_{0}$ on $\C^n$ is equal to the pullback of the function
	\be\label{geneq} w_{0}+\sum_{1\leq i\leq n}-2\beta_i w_i+ \sum_{1\leq i\leq n}\beta_i^2+\beta_{0}\ee
	on $\C^{n+1}$. 
	The theorem follows by applying Theorem~\ref{theorem_top}(3) to the smooth affine variety $i(X)\subset \C^{n+1}$. 
\end{proof}

\begin{remark}
	Notice that we translate the squared Euclidean distance function on $\C^n$ by $\beta_0 \in \C$ to ensure that (\ref{geneq}) is a generic affine linear function on $\C^{n+1}$.
\end{remark}

Let us next show how to compute the ED degree of a  singular affine variety.
We first remark that Theorem~\ref{theorem_top} is a special case of the following result. 

\begin{theorem}\cite[Equation~(2)]{STV}\label{theorem_top_singular}
Let $X$ be an irreducible  closed subvariety in $\C^n$. Let $l: \C^n\to \C$ be a general linear function, and let $H_c$ be the hyperplane in $\C^n$ defined by the equation $l=c$ for a general $c\in \C$. Then
 the number of critical points of $l|_{X_{\reg}}$ is equal to $(-1)^{\dim X}\chi({\rm Eu}_{X}|_{U_c})$, where $U_c=X\setminus H_{c}$ and ${\rm Eu}_X$ is the local Euler obstruction function on $X$.
\end{theorem}

Theorem~\ref{theorem_top_singular}
 can also be derived from  \cite[Theorem 1.2]{TS}, by letting
 $\alpha:=(-1)^{\dim X} {\rm Eu}_X$ and  $k=0$  in their formula (3)
 \footnote{We are grateful to J. Sch\"urmann for bringing the references \cite{STV}, \cite{Tib} and \cite{TS} to our attention.}.
  In fact, the paper \cite{TS} deals with a variant of non-proper Morse theory for complex affine varieties, which uses projective compactifications and transversality at infinity to conclude that there are no singularities at infinity in the context of stratified Morse theory (see also \cite[Chapter 6]{Tib}). 

Furthermore, Theorem~\ref{theorem_top_singular} can be used as in the proof of Theorem~\ref{cor_smooth} to compute the ED degree of a (possibly) singular affine variety as follows.

\bt\label{csing}
Let $X$ be an irreducible closed subvariety of $\C^n$. Then for a general $\bbeta\in \C^{n+1}$	 we have:
\be\label{gen}
\EDdeg(X)=(-1)^{\dim X}\chi({\rm Eu}_X|_{U_{\bbeta}}).
\ee
\et

\begin{remark}
If $X$ is smooth, one has the identity ${\rm Eu}_X=1_X$, so Theorem \ref{cor_smooth} is indeed a special case of Theorem \ref{csing}. 	
\end{remark}


\section{Application to computer vision}\label{cv} 
In this section, we use Theorem~\ref{cor_smooth} to determine the ED degree of a variety coming from computer vision.



\newcommand{\pointP}{y}
We consider a \emph{camera} as a $3\times 4$ matrix of full rank
that defines a linear map from $\pp^3$ to $\pp^2$
sending a point 
$\pointP\in \pp^3$
to its image $A\cdot \pointP\in\pp^2$.
This map is well-defined everywhere except at the kernel of
$A$. 
This kernel corresponds to a point $P_i$ in $\pp^3$ that is called the \emph{center of the camera}.
The \emph{multiview variety} $Y_n$
associated to $n$ cameras
$A_1,A_2,\dots,A_n$ is the closure of the
image of the map
\[
\pp^3\dashrightarrow (\pp^2)^n,\quad \pointP\mapsto(A_1 \cdot \pointP, \dots, A_n \cdot \pointP).
\]
Consider a general affine chart $\cc^{2n}$ of $(\pp^2)^n$.
This amounts to choosing a general affine chart $\cc^2$ of each $\pp^2$.
Define the \emph{affine multiview variety} $X_n$ to be the restriction of $Y_n$ to this chart, i.e., $X_n=Y_n\cap \cc^{2n}$. 
\begin{problem}
If each camera $A_i$ is given by a $3\times 4$ matrix of real numbers, 
then
the data 
\[ 
\balpha=(v_1,w_1,v_2,w_2,\dots,v_n,w_n)\in\rr^{2n}
\]
 represents  
$n$ noisy images $(v_i,w_i)$ of a
point in $\rr^3$
taken by the
$n$ cameras.
The problem of minimizing the squared Euclidean distance function $f_\balpha$ on $X_n$
is called \emph{$n$-view triangulation} in the computer vision literature.
\end{problem}

The paper \cite{DHOST} was motivated to study the ED degree of $X_n$ and solve $n$-view triangulation (see also
\cite{AST2013,HARTLEY1997146,SSN}). 
Our main application of Theorem \ref{cor_smooth} gives a closed form expression for the ED degree of $X_n$
when the cameras are in general position:
\begin{equation}\label{equ_vision}
\EDdeg(X_n)=\frac{9}{2}n^3-\frac{21}{2}n^2+8n-4.
\end{equation} 
This formula was conjecture in \cite[Example 3.3]{DHOST} and it agrees with the computations done in \cite{SSN} with $n=2,3,\dots,7$.

The center of a camera $P\in \bP^3$  defines a natural map $F_P: \Bl_P\bP^3\to \bP^2$ 
where $\Bl_P\bP^3$ is the blowup of $\bP^3$ at $P$. 
Therefore, we have a proper map $F:\Bl_{P_1,\ldots,P_n}\pp^3\to (\pp^2)^n$. 
For the remainder of this section we will assume that $n\geq 3$ and that the cameras are in general position.  
Notice that in this case $F$ is a closed embedding. 
By definition, $Y_n$ is equal to the image of $F$, and hence $Y_n$ is isomorphic to $\Bl_{P_1,\ldots,P_n}\pp^3$. 
%
Let $F_i: Y_n\to \bP^2$ denote the projection of $Y_n$ to the $i$-th factor of $(\bP^2)^n$. 

 
 
 Now we are going to compute the ED degree of the affine multiview variety $X_n$.
Since $X_n$ is a smooth affine variety,
 by Theorem \ref{cor_smooth}  we have that 
  \[\EDdeg(X_n)=-\chi(X_n\cap U_{\bbeta}),\]
  where $\bbeta=(\beta_0, \beta_1, \ldots, \beta_{2n})$ is a general point in $\bC^{2n+1}$ and $U_\bbeta$ is the complement of the hypersurface $\sum_{1\leq i\leq 2n}(z_i-\beta_i)^2+\beta_0=0$ in $\bC^{2n}$. 

Write each $\bP^2$ as $\bC^2\cup \bP^1_{\infty}$, where $\bC^2$ is the chosen affine chart and $\bP^1_{\infty}$ is the line at infinity. 
Denote the hypersurface $\bP^2\times \cdots\times \bP^2\times \bP^1_\infty\times \bP^2\times\cdots\times \bP^2$ in $(\bP^2)^n$ by $H_{\infty, i}$, where $\bP^1_\infty$ is the $i$-th factor. 
Let $H_\infty=\bigcup_{1\leq i\leq n}H_{\infty, i}$. Denote by $H_Q$ the closure of the hypersurface $\sum_{1\leq i\leq 2n}(z_i-\beta_i)^2+\beta_0=0$ in $(\bP^2)^n$. 
In the remaining of this proof, we will use the following notations:
$$D_Q:=Y_n\cap H_Q, \ \ D_{\infty, i}:=Y_n\cap H_{\infty, i}, \ \ D_\infty:=Y_n\cap H_\infty.$$
Notice that $H_\infty$ is the complement of the affine chart $\cc^{2n}$ in $(\pp^2)^n$, thus $D_\infty$ is the complement of $X_n$ in $Y_n$. 

The main result of this section is the following. 
\begin{theorem}\label{thm_vision}
Suppose the blowup points in $\bP^2$, the affine chart in each $\bP^2$ and the $\bbeta\in \bC^{2n+1}$ are general. Then
\begin{enumerate}
\item\label{i1} $\chi(Y_n)=2n+4$;
\item\label{i2} $\chi(D_Q)=4n^3-9n^2+9n$;
\item\label{i3} $\chi(D_\infty)=\frac{n^3}{6}-\frac{3n^2}{2}+\frac{16n}{3}$;
\item\label{i4} $\chi(D_Q\cap D_\infty)=-\frac{n^3}{3}+\frac{13n}{3}$. 
\end{enumerate}
\end{theorem}
Since the Euler characteristic is additive in the complex algebraic setting, Theorem \ref{thm_vision} then yields:
\begin{eqnarray*}
\chi(X_n\cap U_{\bbeta}) 
&=& \chi(X_n\setminus D_Q)\\
&=& \chi\big(Y_n\setminus(D_\infty\cup D_Q)\big)\\
&=& \chi(Y_n)-\chi(D_Q)-\chi(D_\infty)+\chi(D_Q\cap D_\infty)\\ 
&=& -\frac{9}{2}n^3+\frac{21}{2}n^2-8n+4.
\end{eqnarray*}
Therefore, equation (\ref{equ_vision}) follows from Theorem \ref{thm_vision} and Theorem \ref{cor_smooth}. 

The rest of this section will be devoted to prove the above theorem. We will first prove the statements (\ref{i1}), (\ref{i3}), (\ref{i4}). 
\begin{proof}[Proof of Theorem \ref{thm_vision} (\ref{i1}), (\ref{i3}), (\ref{i4})]
We start with statement (\ref{i1}). Since $Y_n$ is isomorphic to the blowup of $n$ points in $\bP^3$, and since blowing up each point increases the Euler characteristic by 2, (\ref{i1}) follows. 

Next, we prove (\ref{i3}). 
Denote the blowup map by $\pi: Y_n\to \bP^3$, and let $E_i=\pi^{-1}(P_i)$ be the divisor in $Y_n$ corresponding to blowing up $P_i$.
By construction, $Y_n\cap H_{\infty, i}$ is the preimage of a general hyperplane in $\bP^2$ under the projection $F_i: Y_n\to \bP^2$. In particular, $D_{\infty, i}=Y_n\cap H_{\infty, i}$ does not intersect $E_j$ for $j\neq i$. Thus, $D_{\infty, i}$ is isomorphic to the blowup of $\bP^2$ at a point, and $\pi(D_{\infty, i})$ is a general hyperplane in $\bP^3$ passing through $P_i$. Therefore, for any distinct pair  $i, j\in \{1, \ldots, n\}$, the intersection $D_{\infty, i}\cap D_{\infty, j}=Y_n\cap H_{\infty, i}\cap H_{\infty, j}$ is isomorphic to $\bP^1$. For any distinct triple $i, j, k\in \{1, \ldots, n\}$, the intersection $D_{\infty, i}\cap D_{\infty, j}\cap D_{\infty, k}=Y_n\cap H_{\infty, i}\cap H_{\infty, j}\cap H_{\infty, k}$ consists of a point. The intersections of $Y_n$ with four or more hyperplanes at infinity are empty. Altogether, we have: 
\begin{itemize}
\item $\chi(D_{\infty, i})=4$;
\item $\chi(D_{\infty, i}\cap D_{\infty, j})=2$;
\item $\chi(D_{\infty, i}\cap D_{\infty, j}\cap D_{\infty, k})=1$.
\end{itemize}
Thus, by the inclusion-exclusion principle for $D_\infty=\bigcup_{1\leq i\leq n}D_{\infty, i}$, we get
$$\chi(D_\infty)=4n-2{n \choose 2}+{n\choose 3}=\frac{n^3}{6}-\frac{3n^2}{2}+\frac{16n}{3}.$$

To prove (\ref{i4}), we recall that $H_Q$ is the closure of the affine hypersurface 
\begin{equation}\label{equ_affine}
\sum_{1\leq i\leq 2n}(z_i-\beta_i)^2+\beta_0=0
\end{equation}
in $(\bP^2)^n$. We introduce homogeneous coordinates $x_i, y_{2i-1}, y_{2i}$ with $z_{2i-1}=\frac{y_{2i-1}}{x_i}$ and $z_{2i}=\frac{y_{2i}}{x_i}$ for $1\leq i\leq n$. Then the homogenization of (\ref{equ_affine}), and hence the equation of $H_Q$, is
\begin{multline}\label{equ_hom}
\big((y_1-\beta_1x_1)^2+(y_2-\beta_2x_1)^2\big)x_2^2\cdots x_n^2+\cdots\\
+x_1^2\cdots x_{n-1}^2\big((y_{2n-1}-\beta_{2n-1}x_n)^2+(y_{2n}-\beta_{2n}x_n)^2\big)+\beta_0 x_1^2\cdots x_n^2=0,
\end{multline}
and the hyperplane $H_{\infty, i}$ is defined by $x_i=0$. Therefore, 
\begin{multline}
H_{\infty, i}\cap H_Q=\\
\{y_{2i-1}+\sqrt{-1}y_{2i}=x_i=0\}\cup \{y_{2i-1}-\sqrt{-1}y_{2i}=x_i=0\}\cup \bigcup_{j\neq i}\{x_i=x_j=0\}.
\end{multline}
We next introduce the following notations:
\begin{align*}
K_i^+&:=Y_n\cap \{y_{2i-1}+\sqrt{-1}y_{2i}=x_i=0\}\\
K_i^-&:=Y_n\cap \{y_{2i-1}-\sqrt{-1}y_{2i}=x_i=0\}\\
L_{i, j}&:=Y_n\cap \{x_i=x_j=0\}, j \neq i.
\end{align*}
By construction,  $K_i^+$ is equal to $F_i^{-1}([0, -\sqrt{-1}, 1])$. Recall that $H_{\infty, i}$ is the strict transformation of a general hyperplane passing through $P_i$. In other words, $\pi(D_{\infty, i})=\pi(Y_n\cap H_{\infty, i})$ is a general hyperplane in $\bP^3$ passing through $P_i$. Then, $\pi(K_i^+)$ is a general line in $\pi(D_{\infty, i})$ passing through $P_i$. So is $\pi(K_i^-)$. Moreover, $\pi(D_{\infty, i})\cap \pi(D_{\infty, j})$ is a line in $\bP^3$, and it is isomorphic to $D_{\infty, i}\cap D_{\infty, j}$ via the map $\pi$. Therefore, we conclude the following:
\begin{enumerate}
\item $K_i^{\pm}\cap K_j^{\pm}=\emptyset$, for $i\neq j$;
\item $K_i^{\pm}\cap L_{i, j}$ consists of a point, for $j\neq i$;
\item $K_i^{\pm}\cap  L_{j, k}=\emptyset$, for distinct $i, j, k$;
\item $ L_{i, j}\cap  L_{i, k}= L_{i, j}\cap  L_{i, k}\cap  L_{j, k}$ consists of a point, for distinct $i, j, k$;
\item $ L_{i, j}\cap  L_{k, l}=\emptyset$, for distinct $i, j, k, l$;
\item $ K_i^{\pm}\cap  L_{i, j}\cap  L_{i, k}=\emptyset$, for distinct $i, j, k, l$. 
\end{enumerate}
Notice that
$$D_Q\cap D_\infty=Y_n\cap H_Q\cap H_\infty=\bigcup_{i}K_i^{+}\cup \bigcup_{i} K_i^{-}\cup \bigcup_{i\neq j} L_{i, j}.$$
Thus, by the inclusion-exclusion principle, we have 
$$\chi\big(D_Q\cap D_\infty\big)=2n+2n+2{n \choose 2}-2n(n-1)-2{n \choose 3}=-\frac{n^3}{3}+\frac{13n}{3}.$$
\end{proof}

The computation of $\chi(D_Q)$ is more difficult, since $D_Q$ is a hypersurface in $Y_n$ with a one-dimensional singular locus. A general formula for computing the Euler characteristic of such singular hypersurfaces is given by the following result. 

\begin{theorem}\cite{PP}, \cite[Theorem 10.4.4]{Max}\label{thm_euler}
Let $X$ be a smooth complex projective variety, and let $V$ be a very ample divisor in $X$. Let $V=\bigsqcup_{S\in \sS} S$ be a Whitney stratification of $X$. Let $W$ be another divisor on $X$ that is linearly equivalent to $V$. Suppose $W$ is smooth and $W$ intersects $V$ transversally in the stratified sense (with respect to the above Whitney stratification). Then we have
\begin{equation}\label{equ_euler}
\chi(W)-\chi(V)=\sum_{S\in \sS}\mu_S \cdot \chi(S\setminus W)
\end{equation}
where $\mu_S$ is the Euler characteristic of the reduced cohomology of the Milnor fiber at any point $x\in S$. 
\end{theorem}
The notion of Milnor fiber is essential in the study of hypersurface singularities. Here we give a precise definition of the constants $\mu_S$, and we refer the readers to \cite[Chapter 10]{Max} for more details. Choose any Whitney stratum $S\in \sS$ and any point $x\in S$. In a sufficiently small ball $B_{\varepsilon,x}$ centered at $x$, the hypersurface $V$ is equal to the zero locus of a holomorphic function $f$. The Milnor fiber of $V$ at $x \in S$ is given by 
$$F_x=B_{\varepsilon,x} \cap \{f=t\}$$
for $0<|t| \ll \varepsilon$. The topological type of the Milnor fiber $F_x$ is independent of the choice of the local equation $f$ at $x$, and it is constant along the given stratum containing $x$; in particular, 
$$\mu_S:=\sum_{k\geq 1} (-1)^k \dim \widetilde{H}^k(F_x; \bQ)$$
is a well-defined intrinsic invariant of the stratum $S$ of $V$. 

\begin{remark}\label{topm} Note that only singular strata of $V$ contribute to the right-hand side of formula (\ref{equ_euler}) since the Milnor fiber at a smooth point is contractible.
\end{remark}

We will prove Theorem \ref{thm_vision} (\ref{i2}) in the following three steps. 
\begin{enumerate}
\item Compute the Euler characteristic of a smooth divisor in $Y_n$ that is linearly equivalent to $D_Q=Y_n\cap H_Q$. 
\item Find the singular locus of $D_Q$ and construct a Whitney stratification $\sS$ of $D_Q$. 
\item Compute the constants $\mu_S$ for each (singular) stratum $S\in \sS$. 
\end{enumerate}

In $(\bP^2)^n$, the hypersurface $H_Q$ is defined by the homogeneous equation (\ref{equ_hom}). From this equation, it is clear that we have a linear equivalence of divisors in $(\bP^2)^n$,
$$H_Q\equiv 2H_{\infty, 1}+\cdots+2H_{\infty, n}.$$
Recall that $Y_n$ is a subvariety of $(\bP^2)^n$ and $F_i$ is the projection from $Y_n$ to the $i$-th factor of $(\bP^2)^n$. Thus, as divisors of $Y_n$, we have
\begin{align*}
D_Q&\equiv 2Y_n\cap H_{\infty, 1}+\cdots+2Y_n\cap H_{\infty, n}\\
&\equiv 2F_1^*(H_{\bP^2})+\cdots+2F_n^*(H_{\bP^2}),
\end{align*}
where $H_{\bP^2}$ denotes a line in $\bP^2$. By the construction of $Y_n$ and $F_i$, 
$$F_i^*(H_{\bP^2})\equiv D_H+E_i$$
where $D_H$ is the preimage of a general hyperplane of $\bP^3$ under $\pi: Y_n\to \bP^3$, and $E_i=\pi^{-1}(P_i)$. Therefore, 
\begin{equation}\label{equ_0}
D_Q\equiv 2nD_H+2E_1+\cdots+2E_n.
\end{equation}
Clearly, $H_Q$ is a very ample divisor in $(\bP^2)^n$, and hence the divisor $D_Q$ in $Y_n$ is also very ample. Therefore, a general divisor, denoted by $D'$, in the linear system $\Gamma(Y_n, \sO(D_Q))$ is smooth.
Thus we have a short exact sequence of vector bundles on $D'$,
\begin{equation}\label{equ_SES}
0\to T_{D'}\to T_{Y_n}|_{D'}\to N_{D'/Y_n}\to 0.
\end{equation}
By the adjunction formula, we have
$$N_{D'/Y_n}\cong \sO_{Y_n}(D')|_{D'}.$$
By (\ref{equ_SES}) and the Whitney sum formula for the total Chern class, we have
$$c(T_{Y_n}|_{D'})=c(T_{D'})c\big(\sO_{Y_n}(D')|_{D'}\big)$$
or equivalently,
\begin{equation}\label{equ_1}
c(T_{D'})=\frac{c(T_{Y_n}|_{D'})}{(1+[D'])|_{D'}}.
\end{equation}
By the standard formula for the total Chern class of the projective space and of a blowup (see, e.g., \cite[Example 3.2.11, Example 15.4.2]{Fulton}), we have 
\begin{equation}\label{equ_2}
c(T_{Y_n})=(1+[D_H])^4+\sum_{1\leq i\leq n}(1+[E_i])(1-[E_i])^3-1.
\end{equation}
Since $D_Q\equiv D'$,  by (\ref{equ_0}), (\ref{equ_1}) and (\ref{equ_2}), we have
\begin{equation}\label{equ_3}
c(T_{D'})=\left.\left((1+[D_H])^4+\sum_{1\leq i\leq n}(1+[E_i])(1-[E_i])^3-1\right)\cdot\left(1+2n[D_H]+2\sum_{1\leq i\leq n}[E_i]\right)^{-1}\right\vert_{D'}.
\end{equation}
Since $[D_H]\cdot [E_i]=0$ and $[E_i]\cdot [E_j]=0$ for any $i\neq j$, (\ref{equ_3}) implies
\begin{equation}
c_2(T_{D'})=\left.(4n^2-8n+6)[D_H]^2\right\vert_{D'}.
\end{equation}
By the Gauss-Bonnet theorem, we have
\begin{equation}
\chi(D')=\int_{D'}\left.(4n^2-8n+6)[D_H]^2\right\vert_{D'}=\int_{Y_n}(4n^2-8n+6)[D_H]^2\cdot [D'].
\end{equation}
Since $[D']=2n[D_H]+2\sum_{1\leq i\leq n}[E_i]$ and since $[D_H]\cdot [E_i]=0$ for any $i$, we have
\begin{equation}
\int_{Y_n}(4n^2-8n+6)[D_H]^2\cdot [D']=\int_{Y_n}(8n^3-16n^2+12n)[D_H]^3=8n^3-16n^2+12n.
\end{equation}
In summary, we have proved the following result:
\begin{proposition}
For a general divisor $D'$ in the linear system $\Gamma(Y_n, \sO(D_Q))$, we have 
\begin{equation}\label{equ_gen}
\chi(D')=8n^3-16n^2+12n.
\end{equation}
\end{proposition}

Recall that the hypersurface $H_Q\subset (\bP^2)^n$ is defined by equation (\ref{equ_hom}), which can be rewritten as
\begin{multline}\label{equ_hom2}
\left(\left(y_1^2+y_2^2\right)x_2^2\cdots x_n^2+\cdots+x_1^2\cdots x_{n-1}^2 \left(y_{2n-1}^2+y_{2n}^2\right)\right)+(\beta_1^2+\cdots+\beta_{2n}^2+\beta_0)x_1^2\cdots x_{2n}^2\\
-2\left((\beta_1 y_1+\beta_2 y_2)x_1x_2^2\cdots x_n^2+\cdots+x_1^2\cdots x_{n-1}^2 (\beta_{2n-1}y_{2n-1}+\beta_{2n}y_{2n})x_n\right)=0
\end{multline}
If we consider $y_{2i-1}, y_{2i}$ and $x_i$ as sections of line bundles $Y_n$, then $D_Q=Y_n \cap H_Q$ is a general divisor in the linear system given by a subspace of 
$$\Gamma(Y_n, \sO(2n D_{H}+2E_1+\cdots+2E_n))$$
generated by the sections
\begin{itemize}
\item $\left(\left(y_1^2+y_2^2\right)x_2^2\cdots x_n^2+\cdots+x_1^2\cdots x_{n-1}^2 \left(y_{2n-1}^2+y_{2n}^2\right)\right)$,
\item $x_1^2\cdots x_{2n}^2$,
\item $y_{2i-1}x_i\cdot x_1^2\cdots \widehat{x_i^2}\cdots x_n^2$ and $y_{2i}x_i\cdot x_1^2\cdots \widehat{x_i^2}\cdots x_n^2$, where \ $\widehat{\cdot}$ \ means that the respective term is removed. 
\end{itemize}
It is easy to see that the base locus of this linear system is equal to $\bigcup_{i< j}D_{\infty, i}\cap D_{\infty, j}$. By the Bertini theorem, $D_Q$ is smooth away from its base locus. On the other hand, $D_Q$ has multiplicity at least 2 along $\bigcup_{i< j}D_{\infty, i}\cap D_{\infty, j}$. Thus, we have proved the following: 
\begin{proposition}
The singular locus of $D_Q$ is equal to
$$\bigcup_{ i< j }D_{\infty, i}\cap D_{\infty, j}.$$
\end{proposition}
Therefore, a Whitney stratification of $D_Q$ must be a refinement of the stratification 
$$D_Q=S_0\sqcup \bigsqcup_{i< j}S_{i, j}\sqcup\bigsqcup_{i< j< k}S_{i, j, k},$$
with
\begin{align*}
S_0&=(D_Q)_{\reg}=D_Q\big\backslash \bigcup_{ i< j }D_{\infty, i}\cap D_{\infty, j},\\
S_{i, j}&=D_{\infty, i}\cap D_{\infty, j}\Big\backslash \bigcup_{k\neq i, k\neq j}D_{\infty, i}\cap D_{\infty, j}\cap D_{\infty,k},\\
S_{i, j, k}&=D_{\infty, i}\cap D_{\infty, j}\cap D_{\infty,k}. 
\end{align*}
Considering $x_i, y_{2i-1}, y_{2i}$ as sections of line bundles on $Y_n$, then $D_Q$ is defined as the zero section of (\ref{equ_hom}). Generically along $S_{i, j}$, the divisor $D_Q$ is analytically isomorphic to the product of a plane node and a disc, or, equivalently, the germ of $xy=0$ at the origin of $\bC^3$ with coordinates $x, y, z$. However, if a point in $S_{i, j}$ satisfies $(y_{2i-1}-\beta_{2i-1}x_i)^2+(y_{2i}-\beta_{2i}x_i)^2=0$ or $(y_{2j-1}-\beta_{2j-1}x_j)^2+(y_{2j}-\beta_{2j}x_j)^2=0$, then the germ of $D_Q$ at that point is isomorphic to the Whitney umbrella, or, equivalently, the germ of $xy^2=z^2$ at the origin of $\bC^3$. When the choice of $\bbeta=(\beta_1, \ldots, \beta_{2n})$ is general, the equation 
$$\left((y_{2i-1}-\beta_{2i-1}x_i)^2+(y_{2i}-\beta_{2i}x_i)^2\right)\left((y_{2j-1}-\beta_{2j-1}x_j)^2+(y_{2j}-\beta_{2j}x_j)^2\right)=0$$
defines 4 distinct points in $S_{i, j}$. Now, denote the subset consisting of these $4$ points in $S_{i, j}$ by $S_{i, j}^1$, and denote its complement in in $S_{i, j}$ by $S_{i, j}^0$. From the equation (\ref{equ_hom}), it is easy to see that $D_Q$ has the same singularity type along $S_{i, j}^0$. Since $S_0$ is the smooth locus of $D_Q$, we have therefore constructed a Whitney stratification of $D_Q$, namely: 
\begin{proposition}\label{prop_whitney}
The stratification 
$$D_Q=S_0\sqcup \bigsqcup_{i< j}S_{i, j}^0 \sqcup \bigsqcup_{i< j}S_{i, j}^1\sqcup \bigsqcup_{i< j< k}S_{i, j, k}$$
is a Whitney stratification of $D_Q$. 
\end{proposition}
Next, we will describe the singularity type of $D_Q$ along each stratum, and compute the Euler characteristic of the reduced cohomology of the corresponding Milnor fiber. The divisor $D_Q$ is smooth along $S_0$. We have argued above that along $S_{i, j}^0$, the divisor $D_Q$ is locally isomorphic to the germ of $xy=0$ at the origin of $\bC^3$. Moreover, along $S_{i, j}^1$, the divisor $D_Q$ is locally isomorphic to the germ of $xy^2=z^2$ at the origin. 

Near $S_{i, j, k}$, the holomorphic functions $x:=\frac{x_i}{y_{2i}}, y:=\frac{x_j}{y_{2j}}$ and $z:=\frac{x_k}{y_{2k}}$ form a coordinate system of $Y_n$. Thus, locally $D_Q$ is defined by an equation of holomorphic functions of the form 
\begin{equation}\label{equ_def}
u_1x^2y^2+u_2x^2z^2+u_3y^2z^2=u_4x^2y^2z^2
\end{equation}
where $u_1, u_2, u_3, u_4$ are locally nonvanishing holomorphic functions. After a coordinate change, equation (\ref{equ_def}) becomes 
\begin{equation}\label{equ_def2}
x^2y^2+x^2z^2+y^2z^2=x^2y^2z^2. 
\end{equation}

\begin{proposition}\label{prop_reduced}
Let $\mu_0$, $\mu_{i, j}^0$, $\mu_{i,j}^1$ and $\mu_{i, j, k}$ be the Euler characteristic of the reduced cohomology of the Milnor fiber of $D_Q$ along the strata $S_0$, $S_{i, j}^0$, $S_{i,j}^1$ and $S_{i, j, k}$, respectively. Then:
\begin{enumerate}
\item\label{j1} $\mu_0=0$;
\item\label{j2} $\mu_{i, j}^0=-1$;
\item\label{j3} $\mu_{i, j}^1=1$;
\item\label{j4} $\mu_{i, j, k}=15$. 
\end{enumerate}
\end{proposition}
\begin{proof}
Since $D_Q$ is smooth along $S_0$, assertion (\ref{j1}) follows from Remark \ref{topm}.
The Milnor fiber of $D_Q$ along the stratum $S_{i, j}^0$ can be described as
$$\{(x,y,z)\in \bC^3 \mid xy=1\} \cong \bC^* \times \bC,$$
hence it is homotopy equivalent to a circle. Thus, (\ref{j2}) follows. 

Since $xy^2-z^2$ is a graded homogeneous function, its local Milnor fiber at the origin is equal to the global Milnor fiber $F_t=\{xy^2-z^2=t\}$ in $\bC^3$. The projection $$F_t\to \bC^2, (x, y, z)\mapsto (x, y)$$
is a proper degree 2 map, with ramification locus $\{xy=t\}$. Thus, the ramification locus is isomorphic to $\bC^*$. Therefore, 
$$\chi(F_t)=2\chi(\bC^2)-\chi(\bC^*)=2.$$
In other words, the reduced Euler characteristic is equal to $1$, and hence (\ref{j3}) follows.
In fact, it can easily be seen, as an application of the Thom-Sebastiani theorem, that the Milnor fiber of the Whitney umbrella at the origin is homotopy equivalent to the $2$-sphere $S^2$, see, e.g., \cite[Example 10.1.21]{Max}. 

Denote by $B_\varepsilon$ the ball at the origin in $\bC^3$ of radius $\epsilon>0$. Define 
$$G_t=\{x^2y^2+x^2z^2+y^2z^2-x^2y^2z^2=t\}\cap B_\epsilon$$
Then for $0<|t| \ll\varepsilon \ll 1$, 
$G_t$ is homeomorphic to the Milnor fiber of $D_Q$ at $S_{i, j, k}$. Notice that there exists a proper degree 8 map
$$\Phi: G_t\to \{xy+xz+yz-xyz=t\}\cap B_{\varepsilon^2}, (x, y, z)\mapsto (x^2, y^2, z^2).$$
Denote $\{xy+xz+yz-xyz=t\}\cap B_{\varepsilon^2}$ by $G'_t$. Then the map $\Phi$ is generically 8-to-1, and ramifies along the disjoint union 
$$\big(G'_t\cap\{x=y=0\}\big)\sqcup \big(G'_t\cap\{x=z=0\}\big)\sqcup \big(G'_t\cap\{y=z=0\}\big). $$
Over the ramification locus, the map $\Phi$ is 2-to-1. Since $G'_t\cap\{x=y=0\}$ is a punctured disc, its Euler characteristic is zero. Thus we have
\begin{equation}\label{equ_8times}
\chi(G_t)
=8\chi(G'_t).
\end{equation}
Notice that the hypersurface $xy+xz+yz-xyz=0$ has an isolated singularity at the origin in $\bC^3$, so the corresponding Milnor fiber $G'_t$ is homotopy equivalent to a bouquet of $2$-spheres. The number of the $2$-spheres in this bouquet is the Milnor number, which can easily be computed (by using the Jacobian ideal) to be $1$. 
Therefore $G'_t\simeq S^2$, so $\chi(G'_t)=2$, By equation (\ref{equ_8times}), we get that $\chi(G_t)=16$, and hence the Euler characteristic of the reduced cohomology of $G_t$ is $15$. 
\end{proof}

We can now complete the proof of Theorem \ref{thm_vision}.
\begin{proof}[Proof of Theorem \ref{thm_vision} (\ref{i2})]
We will apply Theorem \ref{thm_euler} with $V=D_Q$ and $W=D'$, together with the Whitney stratification of Proposition \ref{prop_whitney} and the Milnor fiber calculations of Proposition \ref{prop_reduced}.

By our construction, each stratum $S_{i, j, k}$ consists of a single point and each $S_{i, j}^1$ consists of 4 points. Each $S_{i, j}^0$ is equal to $D_{\infty, i}\cap D_{\infty, j}$ with $n+2$ points removed: 4 points from $S_{i, j}^1$ and $n-2$ points from $S_{i, j, k}$ ($k\neq i, j$). Since $D'$ is chosen to be general, it does not intersect any of the zero-dimensional strata $S_{i, j}^1$ and $S_{i, j, k}$. The intersection $D_{\infty, i}\cap D_{\infty, j}$ is the preimage of a general line in $\bP^3$ under the blowup map $\pi: Y_n\to \bP^3$. Since as a divisor of $Y_n$,
$$D'\equiv 2n D_H+2E_1+\cdots+E_n$$
the divisor $D'$ intersects the line $D_{\infty, i}\cap D_{\infty, j}$ transversally at $2n$ points. Thus, each $S_{i, j}^0\setminus D'$ is equal to a line $\bP^1$ with $3n+2$ points removed, and hence
\begin{equation}\label{last1}
\chi(S_{i, j}^0\setminus D')=-3n. 
\end{equation}
Since each $S_{i, j}^1$ consists of 4 points and $S_{i, j, k}$ consists of one point, and since none of them is contained in $D'$, we have
\begin{equation}\label{last2}
\chi(S_{i, j}^1\setminus D')=4 \textrm{ and } \chi(S_{i, j, k}\setminus D')=1. 
\end{equation}
By plugging (\ref{last1}), (\ref{last2}), and the calculations of Proposition \ref{prop_reduced}
 into equation (\ref{equ_euler}) , we conclude that
\begin{align*}
\chi(D')-\chi(D_Q)&=-\sum_{i< j}\chi(S_{i, j}^0\setminus D')+\sum_{i<j}\chi(S_{i, j}^1\setminus D')+\sum_{i<j<k}15\chi(S_{i, j, k}\setminus D')\\
&=3n{n \choose 2}+4{n \choose 2}+15{n \choose 3}\\
&=4n^3-7n^2+3n. 
\end{align*}
By using equation (\ref{equ_gen}), we get that 
\begin{equation}
\chi(D_Q)=4n^3-9n^2+9n,
\end{equation}
thus proving assertion (\ref{i2}) of Theorem \ref{thm_vision}.
\end{proof}


\bibliographystyle{abbrv}
\bibliography{REF_ED_Degree}

\end{document}